\documentclass[11pt]{amsart}

  \usepackage{amsmath, amsfonts, amssymb,amsthm}

\newtheorem{theorem}{Theorem}
\newtheorem{lemma}{Lemma}
\newtheorem{cor}{Corollary}

\theoremstyle{definition}

\def\min{\mathop{\mathrm{min}}}

\newcommand{\C}{\mathbb{C}}

\newcommand{\N}{\mathbb{N}}

\newcommand{\R}{\mathbb{R}}

\begin{document}
\title{{A uniqueness problem for entire functions related to Br\"uck's conjecture}}
\author{Nguyen Van Thin}

\address{Department of Mathematics, Thai Nguyen University of Education, Luong Ngoc Quyen street, Thai Nguyen city, Viet Nam.}
\email{nguyenvanthintn@gmail.com}
\author{Ha Tran Phuong}

\address{Department of Mathematics, Thai Nguyen University of Education, Luong Ngoc Quyen street, Thai Nguyen city, Viet Nam.}
\email{hatranphuong@yahoo.com}


\begin{abstract}
In this paper, we prove a normal criteria for family of meromorphic functions. As an application of that result, we establish
  a uniqueness theorem for entire function concerning a conjecture of R. Br\"uck. The above uniqueness theorem is an improvement of a problem studied by L. Z. Yang et. al \cite{YZ}. However, our method differs  the method of L. Z. Yang et. al \cite{YZ}. We mainly use normal family theory and combine it with Nevanlinna theory instead of using only the Nevanlinna theory as in \cite{YZ}.
\end{abstract}
\baselineskip=16truept 
\maketitle 
\pagestyle{myheadings}
\markboth{}{}

\section{ Introduction}
\def\theequation{1.\arabic{equation}}
\setcounter{equation}{0} 
Let $D$  be a domain in the complex plane $\C$ and $\mathcal F$ be a family of meromorphic functions in $D.$ The family $\mathcal F$ is said to be normal in $D,$ in the sense of Montel, if for any sequence $\{f_v\}\subset \mathcal F,$ there exists a subsequence $\{f_{v_i}\}$ such that $\{f_{v_i}\}$ converges spherically locally uniformly in $D,$ to a meromorphic function
 or $\infty.$

Let $f$ and $g$ be two nonconstant meromorphic functions. Let $a$ and $b$ be two complex numbers. If $g-b=0$ whenever $f-a=0$, we write $f=a\Rightarrow g=b$. If $f=a\Rightarrow g=b$ and $g=b\Rightarrow f=a$, we write $f=a \Leftrightarrow g=b.$ If $f-a$ and $g-b$ have the same zeros and poles (counting multiplicity), then we denote by $f-a \rightleftharpoons g-b.$

Let $f$ be a meromorphic function in  the complex plane $\C$, we recall that the {\it hyper-order} of $f$ is defined by $$\sigma_2(f)=\limsup_{r \to \infty} \dfrac{\log \log T(r, f)}{\log r}.$$

The following conjecture proposed by  R. Br\"uck \cite{B}.

\noindent{\bf Conjecture.} {\it Let $f$ be a nonconstant entire function such that the hyper-order $\sigma_2(f)$ of $f$ is not a positive integer and $\sigma_2(f)<\infty$. If $f$ and $f'$
share $a$ finite value $a-CM$, then $$\dfrac{f'-a}{f-a}=c,$$ where $c$ is a nonzero constant.}

The conjecture in the case of $a=0$ has been proved by Br\"uck  in \cite{B}. From differential equations
$$ \dfrac{f'-a}{f-a}=e^{z^n},  \hspace{0.3cm}  \dfrac{f'-a}{f-a}=e^{e^{z^n}},$$
we see that this conjecture does not hold if $\sigma_2(f)$ is a positive integer or infinite. 
The conjecture in the case of $f$, a function of finite order, has been proved by Gundersen and Yang in \cite{GY}, 
in the case of $f$, a function of infinite order with $\sigma_2(f)<\dfrac{1}{2}$ has been proved by Chen and Shon in \cite{CS}. 
However, the conjecture in the case $\sigma_2(f)\geqslant\dfrac{1}{2}$ is still open.

It is interesting to ask what happens if $f$ is replaced by $f^n$ in the  Br\"uck's conjecture. In 2008, L. Z. Yang and J. L. Zhang found out a result relating to Br\"uck's conjecture as following.
\begin{theorem}\cite{YZ}\label{th1} Let $f$ be a nonconstant entire function, $n\geqslant 7$ be an integer, and $F=f^n$. If $F$ and $F'$ share 1 $CM$, then $F\equiv F'$ and $f$ assumes the form 
$$f=ce^{{z}/{n}},$$
where $c$ is a nonzero constant.
\end{theorem}

Our result concerning Br\"uck's conjecture are shown as following.
\begin{theorem}\label{th2}
 Let $n\in \N$ and $k, n_i,t_i \in \N^{*}, i= 1, \dots,k $ satisfy one of the following conditions:
\begin{align*}
&1) \ k=1, \ n=0, \ n_1\geqslant t_1+1;\\
&2) \ n\geqslant 1\hspace{0.2cm} \text{or}\hspace{0.2cm}k\geqslant 2, \ n_j \geqslant t_j, \ n+\sum_{j=1}^{k}n_j\geqslant \sum_{j=1}^{k}t_j+2.
\end{align*}
Let $a$ and $b$ be two finite nonzero values and $f$ be a nonconstant entire function. If
 $f^{n+n_1+\dots+n_k}=a\rightleftharpoons f^n(f^{n_1})^{(t_1)}\dots(f^{n_k})^{(t_k)}=b$, then
$$\dfrac{f^n(f^{n_1})^{(t_1)}\dots(f^{n_k})^{(t_k)}-b}{f^{n+n_1+\dots+n_k}-a}=c,$$
where $c$ is a nonzero constant. Specially, if $a=b$ then $f=c_1e^{tz}$, where $c_1$ and $t$ are nonzero constants and $t$ is satisfied by $(tn_1)^{t_1}\dots(tn_k)^{t_k}=1$.
\end{theorem}

As a special case, if we take $n=0, \ k=1, \ t_1=1$ in Theorem \ref{th2},  then we have:

\begin{cor} \label{cl1} Let $f$ be a nonconstant entire function, $n \geqslant 2$ be an integer, and $F=f^{n}$. 
If $F$ and $F'$ share 1 $CM$, then $F\equiv F'$, and $f$ assumes the form 
$$f=ce^{{z}/{n}},$$
where $c$ is a nonzero constant.
\end{cor}

Note that, the condition of $n$ in Colorrary \ref{cl1} is $n \geqslant 2$, and in Theorem \ref{th1} is $n \geqslant 7$. Thus Theorem \ref{th2} is an improvement of Theorem 1 of Yang and Zhang. In order to prove Theorem \ref{th2}, we need to use the following result about normal family of meromorphic functions.

\begin{theorem}\label{th3}
 Let $\mathcal F$ be a family of meromorphic functions in a complex domain $D$. Let
 $a $ and $b$ be two complex numbers such that $b\ne 0,$ let $n \in \N$, $n_j, t_j, k \in \N^{*}$, $(j=1,2,\dots,k)$ satisfy
\begin{align}\label{them1}
n_j\geqslant t_j, n+\sum_{j=1}^{k}n_j\geqslant \sum_{j=1}^{k}t_j+3,
\end{align}
and
\begin{align}\label{them2}
f^{n+n_1+\dots+n_k}&=a\Leftrightarrow f^n(f^{n_1})^{(t_1)}\dots(f^{n_k})^{(t_k)} =b
\end{align}
for all $f \in \mathcal F.$ Then $\mathcal F$ is a normal family. Furthermore, if $\mathcal F$ is a family of holomorphic functions, then the statement holds when (\ref{them1}) is replaced by one of the following conditions:
\begin{align}
& k=1, \ n=0, \ n_1\geqslant t_1+1;\\
&n\geqslant1  \hspace{0.2cm}\text{or} \hspace{0.2cm} k\geqslant 2, n_j\geqslant t_j, n+\sum_{j=1}^{k}n_j\geqslant \sum_{j=1}^{k}t_j+2.
\end{align}
\end{theorem}

\section{ Some Lemmas}
\def\theequation{2.\arabic{equation}}
\setcounter{equation}{0} 
In order to prove the above theorems, we need the following lemmas.
\begin{lemma}[Zalcman's Lemma] \cite{Z} \label{L1}  Let $\mathcal F$ be a family of meromorphic functions defined in the open unit disc $\bigtriangleup = \{ z \in \mathbb C: |z| < 1\}.$  Then if $\mathcal F$ is not normal at a point $z_0\in\bigtriangleup,$ there exist, for each real number $\alpha$ satisfying $-1<\alpha<1,$

$1)$ a real number $r,\;0<r<1$ and points $z_n,\;|z_n|<r,$ $z_n\to z_0,$

$2)$ positive numbers $\rho_n,\rho_n\to 0^+,$

$3)$ functions $f_n,\;f_n\in\mathcal F$ such that
$$g_n(\xi)=\frac{f_n(z_n+\rho_n\xi)}{\rho_n^\alpha}\to g(\xi)$$
spherically uniformly on compact subsets of $\C,$ where $g(\xi)$ is a non-constant meromorphic function and $g^{\#}(\xi)\leqslant g^{\#}(0)=1.$ Moreover, the order of $g$ is not greater than $2.$ Here, as usual, $g^\#(z)=\frac{|g'(z)|}{1+|g(z)|^2}$ is the spherical derivative.
\end{lemma}
\begin{lemma}\cite{CH}\label{L2}
Let $g$ be an entire function and $M$ is a positive constant. If $g^{\#}(\xi)\leqslant M$ for all $\xi\in\C,$ then $g$ has the order at most one.
\end{lemma}

\noindent{\bf Remark.} 
In Lemma 1, if $\mathcal F$ is a family of holomorphic functions, then $g$ is a holomorphic function based on Hurwitz's theorem. Therefore, the order of $g$ is not greater than one according to Lemma 2.

We consider a nonconstant meromorphic function $g$ in the complex plane $\C,$ and its first $p$ derivatives. A differential polynomial $P$ of $g$ is defined by 
 $$
P(z):=\sum_{i=1}^n\alpha_i(z)\prod_{j=0}^p(g^{(j)}(z))^{S_{ij}},$$ where $S_{ij}, \ 0\leqslant i,j\leqslant n,$ are nonnegative integers, and $\alpha_i,1 \leqslant i\leqslant n$ are small meromorphic functions with respect to $g$.
Set $$d(P):=\min_{1\leqslant i\leqslant n}\sum_{j=0}^pS_{ij}\;\text{and}\; \theta(P):=\max_{1\leqslant i\leqslant n}\sum_{j=0}^pjS_{ij}.$$ 

In 2002, J. Hinchliffe \cite{Hi} generalized the theorems of Hayman \cite{Ha} and Chuang \cite{Ch} and obtained the following result.

\begin{lemma}\cite{Hi}\label{lm2.3} Let $g$ be a transcendental meromorphic function and $a$ be a nonzero complex constant, let $P$ be a
nonconstant differential polynomial in $g$ with $d(P)\geqslant 2.$ Then
\begin{align*}
T(r,g)\leqslant\frac{\theta(P)+1}{d(P)-1}\overline{N}(r,\frac{1}{g})+\frac{1}{d(P)-1}\overline{N}(r,\frac{1}{P-a})+o(T(r,g)),
\end{align*}
for all $r\in[1,+\infty)$ excluding a set of finite Lebesgues measure. When $f$ is a transcendental entire function, the above inequality becomes
\begin{align*}
T(r,g)\leqslant\frac{\theta(P)+1}{d(P)}\overline{N}(r,\frac{1}{g})+\frac{1}{d(P)}\overline{N}(r,\frac{1}{P-a})+o(T(r,g)), 
\end{align*}
for all $r\in[1,+\infty)$ excluding a set of finite Lebesgues measure. 
\end{lemma}

\begin{lemma}\label{lm2.4}
 Let $f$ be a transcendental meromorphic function and $a$ be a nonzero complex constant. Let $n\in \N$, $k, n_j, t_j \in \N^{*}$, $j=1,\dots,k$ satisfy
$$n+\sum_{j=1}^{k}n_j\geqslant \sum_{j=1}^{k}t_j+3.$$ 
Then the equation
$$f^n(f^{n_1})^{(t_1)}\dots(f^{n_k})^{(t_k)}=a$$ has infinite solutions.
Furthermore, if $f$ is a transcendental entire function, the statement holds when $n+\sum_{j=1}^{k}n_j\geqslant \sum_{j=1}^{k}t_j+2.$
\end{lemma}
\begin{proof} Set
$$P(f)=f^n(f^{n_1})^{(t_1)}\dots(f^{n_k})^{(t_k)}.$$
 It is easy to check 
$d(P)=n+\sum_{j=1}^kn_j$ and $\theta(P)=\sum_{j=1}^kt_j.$ Using Lemma \ref{lm2.3} with $f$ and $P(f)$, we have
\begin{align*}
T(r,f)\leqslant\frac{\sum_{j=1}^kt_j+1}{n+\sum_{j=1}^kn_j-1}\overline{N}(r,\frac{1}{f})+\frac{1}{n+\sum_{j=1}^kn_j-1}\overline{N}(r,\frac{1}{P-a})+o(T(r,f)).
\end{align*}
Since $n+\sum_{j=1}^{k}n_j\geqslant \sum_{j=1}^{k}t_j+3,$ we obtain that the equation $$f^n(f^{n_1})^{(t_1)}\dots(f^{n_k})^{(t_k)}=a$$ has infinite solutions.
Furthermore, if $f$ is a transcendental entire function, we have  
\begin{align*}
T(r,f)\leqslant\frac{\sum_{j=1}^kt_j+1}{n+\sum_{j=1}^kn_j}\overline{N}(r,\frac{1}{f})+\frac{1}{n+\sum_{j=1}^kn_j}\overline{N}(r,\frac{1}{P-a})+o(T(r,f)).
\end{align*}
So the condition $n+\sum_{j=1}^{k}n_j\geqslant \sum_{j=1}^{k}t_j+2$
implies that 
$$f^n(f^{n_1})^{(t_1)}\dots(f^{n_k})^{(t_k)}=a$$ has infinite solutions.
\end{proof} 
\begin{lemma}\label{lm2.5}
Let $f$ be a nonconstant rational function and $a$ be a nonzero complex constant. Let $n\in \N$, $k, n_j, t_j \in \N^{*}$, $j=1,\dots,k$ satisfy
$$n_j\geqslant t_j, n+\sum_{j=1}^{k}n_j\geqslant \sum_{j=1}^{k}t_j+2, j=1,\dots,k.$$ 
Then the equation
$$f^n(f^{n_1})^{(t_1)}\dots(f^{n_k})^{(t_k)}=a$$ has at least two distinct zeros.
\end{lemma}
\begin{proof}
We consider some cases as following.

\noindent {\bf Case 1.} $f$ is a polynomial. Then, we see that $f^n(f^{n_1})^{(t_1)}\dots(f^{n_k})^{(t_k)}$ is a polynomial. 
We suppose that $f^n(f^{n_1})^{(t_1)}\dots(f^{n_k})^{(t_k)}-a$ has a unique zero $z_0$, so 
$$f^{n}(f^{n_1})^{(t_1)}\dots(f^{n_k})^{(t_k)}-a=A(z-z_{0})^{l},\ l\geqslant2,$$
where $A$ is a nonzero constant. Then $$(f^{n}(f^{n_1})^{(t_1)}\dots (f^{n_k})^{(t_{k})})'=Al(z-z_0)^{l-1}.$$ 
It implies that $z_0$ is the unique zero of $(f^n(f^{n_1})^{(t_1)}\dots(f^{n_k})^{(t_k)})'$. We know that any zero of $f$ is a zero of $f^n(f^{n_1})^{(t_1)}\dots(f^{n_k})^{(t_k)}$ with multiplicity at least 2, and then it is a zero of  $(f^n(f^{n_1})^{(t_1)}\dots(f^{n_k})^{(t_k)})'.$ It leads to that $z_0$ is the unique zero of $f$. We see that
$$0=f^{n}(z_{0})(f^{n_{1}})^{(t_{1})}(z_{0})\dots(f^{n_{k}})^{(t_{k})}(z_{0})=a\ne 0.$$ This is a contradiction. We 
conclude that$$f^n(f^{n_1})^{(t_1)}\dots(f^{n_k})^{(t_k)}=a$$ has at least two distinct zeros.

\noindent{\bf Case 2.} $f$ is a rational function which is not a polynomial.

\noindent {\bf Case 2.1.} $f$ has a zero. Then $f$ can be written as
\begin{align}\label{ct1}
f=A\frac{\prod_{i=1}^{s}(z-\alpha_i)^{m_i}}{\prod_{l=1}^{t}(z-\beta_l)^{d_l}} , \ {{m}_{i}}\geqslant 1, {{d}_{l}}\geqslant 1, i=1,\dots, s, l=1,\dots,t.
\end{align}
Put $M=m_1+\dots+m_s \geqslant s, \ N=d_1+\dots+d_t \geqslant t$. We have
\begin{align}\label{ct2}
f^{n_j}=A^{n_j}\frac{\prod_{i=1}^{s}(z-\alpha_i)^{n_jm_i}}{\prod_{l=1}^{t}(z-\beta_l)^{n_jd_l}},  j=1,\dots, k.
\end{align}
Hence 
\begin{align}\label{ct3}
(f^{n_j})^{(t_j)}=A^{n_j}\frac{\prod_{i=1}^{s}(z-\alpha_i)^{n_jm_i-t_j}}{\prod_{l=1}^{t}(z-\beta_l)^{n_jd_l+t_j}}g_j(z),  
\end{align}
where $g_j$ is a polynomial with $\deg g_j(z)\leqslant t_j(s+t-1)$, $j=1,\dots,k$.
Combine (\ref{ct1}), (\ref{ct2}) and (\ref{ct3}), we get
\begin{align}\label{ct4}
 f^n(f^{n_1})^{(t_1)}\dots (f^{n_k})^{(t_k)}&=
\frac{\prod_{i=1}^{s}(z-\alpha_i)^{(n+\sum_{j=1}^{k}n_j)m_i-\sum_{j=1}^{k}t_j}}
{\prod_{l=1}^{t}(z-\beta_l)^{(n+\sum_{j=1}^{k}n_j)d_l+\sum_{j=1}^{k}t_j}}g(z)\\
 & =\frac{P(z)}{Q(z)} \notag,
\end{align} 
where $g(z)=A^{n+\sum_{j=1}^{k}n_j}\prod_{v=1}^{k}g_v(z)$ with $\deg g(z)\leqslant(\sum_{j=1}^{k}t_j)(s+t-1).$

We suppose that $$f^n(f^{n_1})^{(t_1)}\dots(f^{n_k})^{(t_k)}=a$$ has a unique zero $z_0$. Then $z_0\ne \alpha _i,\text{ }i=1,\dots,s$. Indeed, if $z_0=\alpha_i$ for some $i\in\{1,\dots,s\}.$ We deduce that
$$0=f^n(z_0)(f^{n_1})^{(t_1)}(z_0)\dots(f^{n_k})^{(t_k)}(z_0)=a\ne 0.$$ 
This is a contradiction. We have
\begin{align}\label{ct5}
f^n(f^{n_1})^{(t_1)}\dots(f^{n_k})^{(t_k)}=a+\frac{B(z-z_0)^{l}}{\prod_{l=1}^{t}(z-\beta_l)
^{(n+\sum_{j=1}^{k}n_j)d_l+\sum_{j=1}^{k}t_j}},
\end{align}
where $B$ is a nonzero constant.
It implies
\begin{align}\label{ct6}
(f^{n}(f^{n_1})^{(t_1)}\dots(f^{n_k})^{(t_k)})'=\frac{(z-z_0)^{l-1}G_1(z)}{\prod_{l=1}^{t}(z-\beta_l)^
{(n+\sum_{j=1}^{k}n_j)d_l+\sum_{j=1}^{k}t_j+1}},
\end{align}
where $$G_1(z)=B(l-(n+\sum_{j=1}^{k}n_j)N-(\sum_{j=1}^{k}t_j)t)z^{t}+b_1z^{t-1}+\dots+b_t.$$ 
From (\ref{ct4}), we see
\begin{align}\label{ct97}
(f^{n}(f^{n_1})^{(t_1)}\dots(f^{n_k})^{(t_k)})'=\frac{\prod_{i=1}^{s}(z-\alpha_i)^{(n+\sum_{j=1}^{k}n_j)m_i+\sum_{j=1}^{k}t_j-1}G_2(z)}
{\prod_{l=1}^{t}(z-\beta_1)^{(n+\sum_{j=1}^{k}n_j)d_l+\sum_{j=1}^{k}t_v+1}}.
\end{align}
It is easy to test 
$$s+t-1\leqslant\deg {{G}_{2}}(z)\leqslant(\sum\limits_{j=1}^{k}{{{t}_{j}}}+1)(s+t-1).$$
We consider the following subcases:

\noindent{\bf Case 2.1.1.} $l\ne (n+\sum_{j=1}^{k}n_j)N+(\sum_{j=1}^{k}t_j)t,$ consequently
 $\deg P(z)\geqslant\deg Q(z)$. From (\ref{ct4}), we get
$$\sum_{i=1}^{s}((n+\sum_{j=1}^{k}n_j)m_{i}
-\sum_{j=1}^{k}t_j)+\deg g\geqslant\sum_{j=1}^{t}((n+\sum_{j=1}
^{k}n_j)d_{j}+\sum_{j=1}^{k}t_{j}).$$
We note that $\deg g(z)\leqslant(\sum_{j=1}^{k}t_{j})(s+t-1)$. It leads to 
$$M\geqslant N+\frac{\sum_{j=1}^{k}t_{j}}{n+\sum_{j=1}^{k}n_{j}},$$
 then $M>N.$ Since $z_0\ne\alpha_i$, for all $i=1,\dots,s$, we obtain
$$\sum_{i=1}^{s}((n+\sum_{j=1}^{k}n_j)m_i-\sum_{j=1}^{k}t_j-1)\leqslant\deg G_{1}=t.$$
Consequently,
\begin{align}\label{ct8}
(n+\sum\limits_{j=1}^{k}n_{j})M\leqslant(1+\sum\limits_{j=1}^{k}t_{j})s+t<
(\sum\limits_{j=1}^{k}t_{j}+2)M.
\end{align}
We note that $n+\sum\limits_{j=1}^{k}n_{j}\geqslant\sum\limits_{j=1}^{k}t_{j}+2,$ thus (\ref{ct8}) leads to a contradiction.

\noindent{\bf Case 2.1.2.} $l=(n+\sum\limits_{j=1}^{k}{n}_j)N+(\sum\limits_{j=1}^{k}t_{j})t.$

If $M>N$, then we have a contradiction by the argument as Case 1.  

If $M\leqslant N$. Since $$l-1\leqslant\deg {{G}_{2}}\leqslant(\sum\limits_{j=1}^{k}{{{t}_{j}}}+1)(s+t-1),$$
then
\begin{align}\label{ct9}
(n+\sum\limits_{j=1}^{k}{t{}_{j}})N&=l-(\sum\limits_{j=1}^{k}{t{}_{j}})t\leqslant\deg {{G}_{2}}+1
-(\sum\limits_{j=1}^{k}{t{}_{j}})t\notag\\
&\quad<(1+\sum\limits_{j=1}^{k}{t{}_{j}})s+t\leqslant(\sum\limits_{j=1}^{k}{t{}_{j}}+2)N.
\end{align}
Since the condition $n+\sum\limits_{j=1}^{k}n_{j}\geqslant\sum\limits_{j=1}^{k}t_{j}+2$ and (\ref{ct9}), we get a contradiction.

\noindent{\bf Case 2.2.}  $f$ has not any zero. Then $f$ can be written as 
\begin{align}\label{ct9a}
f=\frac{A}{\prod_{l=1}^{t}(z-\beta_l)^{d_l}}, \ d_l \geqslant1, \  l=1,\dots, t.
\end{align}
Thus, (\ref{ct3}) becomes
\begin{align}\label{ct9b}
(f^{n_j})^{(t_j)}=\frac{A^{n_j}}{\prod_{l=1}^{t}(z-\beta_l)^{n_jd_l+t_j}}g_j(z),
\end{align}
where $g_j$ is a polynomial with $\deg {{g}_j}(z)\leqslant{{t}_{j}}(t-1)$, $j=1,\dots, k$. We have
\begin{align}\label{ct4aa}
 f^n(f^{n_1})^{(t_1)}\dots (f^{n_k})^{(t_k)}&=\frac{g(z)}{\prod_{l=1}^{t}(z-\beta_l)^{(n+\sum_{j=1}^{k}n_j)d_l+\sum_{j=1}^{k}t_j}}\\
 & =\frac{g(z)}{Q(z)}, \notag
\end{align} 
where $g(z)=A^{n+\sum_{j=1}^{k}n_j}\prod\limits_{v=1}^{k}g_{v}(z)$ with $\deg g(z)\leqslant(\sum\limits_{j=1}^{k}t_{j})(t-1)$.
We see that
\begin{align}\label{m1}
f^n(f^{n_1})^{(t_1)}\dots (f^{n_k})^{(t_k)}-a=\dfrac{g(z)-aQ(z)}{Q(z)}.
\end{align}
Since $N=d_1+\dots+d_t\geqslant t,$ it implies
$$\deg Q\geqslant(n+\sum_{j=1}^{k}n_j+\sum_{j=1}^{k}t_j)t>\deg g,$$ thus equation (\ref{m1}) has at least one solution. We suppose that $$f^n(f^{n_1})^{(t_1)}\dots(f^{n_k})^{(t_k)}=a$$ has a unique zero $z_0$. 
We have
\begin{align}\label{ct5aa}
f^n(f^{n_1})^{(t_1)}\dots(f^{n_k})^{(t_k)}=a+\frac{B(z-z_0)^{l}}{\prod_{l=1}^{t}(z-\beta_l)
^{(n+\sum\limits_{j=1}^{k}n_j)d_l+\sum\limits_{j=1}^{k}t_j}},
\end{align}
where $B$ is a nonzero constant.
It implies that
\begin{align}\label{ct6aa}
(f^{n}(f^{n_1})^{(t_1)}\dots(f^{n_k})^{(t_k)})'=\frac{(z-z_0)^{l-1}G_1(z)}{\prod_{l=1}^{t}(z-\beta_l)^
{(n+\sum_{j=1}^{k}n_j)d_l+\sum_{j=1}^{k}t_j+1}},
\end{align}
where $$G_1(z)=B(l-(n+\sum\limits_{j=1}^{k}n_{j})N-(\sum\limits_{j=1}^{k}t_j)t)z^{t}+b_1z^{t-1}+\dots+b_{t}.$$
From (\ref{ct4aa}), we see
\begin{align}\label{ct97aa}
(f^{n}(f^{n_1})^{(t_1)}\dots(f^{n_k})^{(t_k)})'=\frac{G_2(z)}
{\prod_{l=1}^{t}(z-\beta_1)^{(n+\sum_{j=1}^{k}n_j)d_l+\sum_{j=1}^{k}t_j+1}}.
\end{align}
It is easy to test 
$$t-1\leqslant\deg {{G}_{2}}(z)\leqslant(\sum\limits_{j=1}^{k}{{{t}_{j}}}+1)(t-1).$$
We consider the following subcases:

\noindent{\bf Case 2.2.1.} $l\ne (n+\sum\limits_{j=1}^{k}n_j)N+(\sum\limits_{j=1}^{k}t_j)t,$ consequently
 $\deg g(z)\geqslant\deg Q(z)$. From (\ref{ct4aa}), we get
$$\deg g\geqslant\sum_{j=1}^{t}((n+\sum_{j=1}
^{k}n_j)d_j+\sum_{j=1}^{k}t_j)=(n+\sum_{j=1}^{k}n_j)N+(\sum_{j=1}^{k}t_j)t.$$
We note that $\deg g(z)\leqslant(\sum_{j=1}^{k}t_j)(t-1)$. This is a contradiction.

\noindent{\bf Case 2.2.2.} $l=(n+\sum\limits_{j=1}^{k}n_j)N+(\sum\limits_{j=1}^{k}t_j)t.$
Since $$l-1\leqslant\deg {{G}_{2}}\leqslant(\sum_{j=1}^{k}t_j+1)(t-1),$$
then
\begin{align}\label{ct9aa}
(n+\sum_{j=1}^{k}n_j)N&=l-(\sum_{j=1}^{k}t_j)t\leqslant\deg G_2+1
-(\sum_{j=1}^{k}t_j)t\\
&=t-\sum_{j=1}^{k}t_j. \notag
\end{align}
Since $n+\sum_{j=1}^{k}n_j\geqslant\sum_{j=1}^{k}t_j+2$ and $t \leqslant N,$ we have
$$ (\sum_{j=1}^{k}t_j+2)N+\sum_{j=1}^{k}t_j \leqslant N.$$
This is a contradiction. Thus, we obtain that 
$$f^n(f^{n_1})^{(t_1)}\dots(f^{n_k})^{(t_k)}=a$$ has at least two distinct zeros.
\end{proof}

We recall that the order $\sigma(f)$ of meromorphic function $f$ is defined by
$$\sigma(f)=\limsup_{r \to \infty} \dfrac{\log T(r,f)}{\log r}.$$
Furthermore, when $f$ is an entire function, we have  
$$ \sigma(f)=\limsup_{r \to \infty} \dfrac{\log T(r,f)}{\log r}=\limsup_{r \to \infty} \dfrac{\log \log( M(r,f))}{\log r}.$$

Let $f$ be an entire function. We know that $f$ can be expressed by the power series $f(z)=\sum_{n=0}^{\infty}a_nz^n.$
We denote by
$$\mu(r,f)=\max_{n\in \N, |z|=r}\{|a_nz^n|\}, \  \nu(r,f)=\sup\{n: |a_n|r^n=\mu(r,f)\},$$
$$M(r,f)=\max_{|z|=r}|f(z)|.$$

\begin{lemma}\cite{IL}\label{lm2.6}
 If $f$ is an entire function with the order $\sigma(f)$, then
$$\sigma(f)=\limsup_{r \to \infty} \dfrac{\log \nu(r,f)}{\log r}.$$
\end{lemma}

\begin{lemma}\cite{IL}\label{lm2.7}
Let $f$ be a transcendental entire function, let $0<\delta<\dfrac{1}{4}$ and $z$ be such that $|z|=r$ and that
$$|f(z)|>M(r,f)\nu(r,f)^{-\dfrac{1}{4}+\delta}$$
hold. Then there exists a set $F\subset \R_{+}$ of finite logarithmic measure, that is $\int\limits_{F}{\frac{dt}{t}}<+\infty,$ such that 
$$\dfrac{f^{(m)}(z)}{f(z)}=\biggl(\dfrac{\nu(r,f)}{z}\biggl)^{m}(1+o(1))$$
holds for all $m\geqslant 1$ and $r\not\in F.$
\end{lemma}
Taking $E_0(z)=1-z,$ $E_m(z)=(1-z)e^{z+z^2/2+\dots+z^m/m}, m\in \mathbb Z^{+},$ then we have a following result called the Weierstrass Factorization Theorem.
\begin{lemma}\cite{IL}\label{lm2.8}
Let $f$ be an entire function, with a zero multiplicity $m\geqslant0$ at $z=0.$ Let the other zeros of $f$ be at $a_1, a_2, \dots,$ each zero being repeated as many times as its multiplicity implies. Then $f$ has the representation
$$ f(z)=e^{g(z)}z^m\prod_{n=1}^{\infty}E_{m_n}\big(\dfrac{z}{a_n}\big),$$
for some entire function $g$ and some integers $m_n.$ If $\{a_n\}_{n\in \mathbb N}$ has a finite exponent of convergence $\lambda,$ then $m_n$ may be taken as $k=[\lambda]>\lambda-1.$ Furthermore, if $f$ has finite order $\rho,$ then $g$ is a polynomial with degree at most $\rho.$
\end{lemma}

\section{Proofs of Theorems}
\def\theequation{3.\arabic{equation}}
\setcounter{equation}{0} 

\begin{proof}[Proof of Theorem \ref{th3}]

Without loss of the generality, we may assume that $D$ is the unit disc. Suppose that $\mathcal F$ is not normal at $z_0\in D.$ Using Lemma \ref{L1} with $\alpha=\dfrac{\sum_{j=1}^{k}t_j}{n+\sum_{j=1}^{k}n_j},$ we have
$$g_v(\xi)=\frac{f_v(z_v+\rho_v\xi)}{\rho_v^{\alpha}}\to g(\xi)$$ spherically uniformly on compact subsets of $\C,$ where $g(\xi)$ is a non-constant meromorphic function. It implies that
\begin{align*}
f_v^n(z_{v}+\rho _{v}\xi )&(f_v^{n_1})^{(t_1)}(z_{v}+\rho_{v}\xi )
\dots(f_v^{n_k})^{(t_{k})}(z_{v}+\rho_{v}\xi )-b\\ 
&=g_v^{n}(\xi )(g_v^{n_1}(\xi ))^{(t_1)}\dots (g_v^{n_{k}}(\xi ))^{(t_{k})}-b.
\end{align*}
Then we see that
\begin{align}\label{ctm1}
f_{v}^{n}(z_{v}+\rho_{v}\xi )&(f_{v}^{n_1})^{(t_1)}(z_v+\rho_v\xi )
\dots(f_v^{n_k})^{(t_k)}(z_v+\rho_v\xi )-b\\
&\to g^{n}(\xi )({g^{n_{1}}(\xi )})^{(t_{1})}\dots ({g^{n_{k}}(\xi )})^{(t_{k})}-b \notag
\end{align}
 uniform (with metric spherical) on each compact subset of $\C \setminus \{\text{pole } g\}.$

We consider two cases:

\noindent {\bf Case 1.} $a\ne 0$. Let $M$ be a positive constant such that $M\leqslant |a|^{\dfrac{1}{n+n_1+\dots+n_k}}.$ For each $f\in \mathcal F,$ we denote $E_f$ by
$${{E}_{f}}= \Big\{z\in D: f^{n}(f^{n_1})^{(t_1)}\dots(f^{n_k})^{(t_k)}=b\Big\}.$$
Then $|f(z)|\geqslant M$ for any $f\in\mathcal F$ whenever $z\in E_f.$ 

We see that the equation 
\begin{align}\label{ct00}
g^{n}(\xi )({g^{n_{1}}(\xi )})^{(t_{1})}\dots ({g^{n_{k}}(\xi )})^{(t_{k})}=b
\end{align}
has at least a zero $\xi_0$. Indeed, we consider some following subcases.

\noindent {\bf Case 1.1.} $g$ is a meromorphic function. 

If $g$ is a transcendental meromorphic function, we see that the
 equation (\ref{ct00}) has infinite solutions by Lemma \ref{lm2.4}. If $g$ is a rational function, the equation (\ref{ct00}) has at least one zero by Lemma \ref{lm2.5}.

\noindent{\bf Case 1.2.} $g$ is an entire function.

\noindent{\bf Case 1.2.1.} If $g$ is a transcendental entire function.

If $n=0, k=1, n_1= t_1+1$ (see \cite{HW}) and $n_1\geqslant t_1+2$ (by Lemma \ref{lm2.4} and Lemma \ref{lm2.5}), then $(g^{n_1})^{t_1}-b$
 has infinite zeros.

If $ n\geqslant1 \hspace{0.2cm} \text{or} \hspace{0.2cm} k\geqslant2, n_j\geqslant t_j, n+\sum_{j=1}^{k}n_j\geqslant \sum_{j=1}^{k}t_j+2$, by Lemma \ref{lm2.4}, we get that (\ref{ct00}) has infinite 
solutions.

\noindent {\bf Case 1.2.2.} If $g$ is a polynomial. Since $k,n,n_j,t_j$ satisfy the assumption of Theorem 3, then equation (\ref{ct00}) has at least one solution.

To sum up, there exists $\xi_0\in \mathbb C$ satisfying
\begin{align}\label{ct00a}
g^{n}(\xi_0 )({g^{n_{1}}})^{(t_{1})}(\xi_0 )\dots ({g^{n_{k}}})^{(t_{k})}(\xi_0 )=b.
\end{align}
We see that $g(\xi_0)\ne 0, \infty$, so $g_v(\xi)$ converges uniformly to $g(\xi)$ in a neighborhood of $\xi_0$. From (\ref{ctm1}) and Hurwitz's theorem, there exists a sequence ${{\xi }_{v}}\to {{\xi }_{0}}$ such that 
\begin{align*}
 f_{v}^{n}(z_v+\rho_v\xi_v)(f_v^{n_1})^{(t_1)}(z_v
+\rho_v\xi_v)\dots(f_v^{n_k})^{(t_k)}(z_v+\rho_v\xi_v) =b
\end{align*}
for any large number $v$ and ${{\zeta }_{v}}={{z}_{v}}+{{\rho }_{v}}{{\xi }_{v}},$ so $\zeta_v \in E_{f_v}$. It implies that
\begin{align}\label{ct32}
|g_v(\xi_v)|=\dfrac{|f_v(\zeta_v)|}{{\rho }_{v}^{\alpha}}\geqslant \dfrac{M}{{\rho }_{v}^{\alpha}}.
\end{align}
Since $\xi_0$ is not a pole of $g$, then $g(\xi)$ is bounded in a neighborhood $\xi_0$. Taking $v \to \infty$ in (\ref{ct32}), we 
get a contradiction.

\noindent {\bf Case 2.} $a=0$. For any $f \in F,$ if there exists $z_0\in \C$ such that $f(z_0)=0$, then $$f^n(z_0)(f^{n_1})^{(t_1)}(z_0)\dots(f^{n_k})^{(t_k)}(z_0)=0.$$
 Since $b\ne0$, it is a contradiction. Hence $f\ne 0.$ Furthermore, if $$f^n(z_0)(f^{n_1})^{(t_1)}(z_0)\dots(f^{n_k})^{(t_k)}(z_0)=b,$$ for some $z_0\in D$ then $f(z_0)^{n+n_1+\dots+n_k}=0$, so $f(z_0)=0$, thus $b=0$. It is still a contradiction. 

Hence $f\ne 0$ and $f^n(f^{n_1})^{(t_1)}\dots(f^{n_k})^{(t_k)}\ne b$ for all $f\in \mathcal F.$ By Hurwitz's theorem, we have 
$g\ne 0$, $g^n(g^{n_1})^{(t_1)}\dots(g^{n_k})^{(t_k)}\ne b$ or
 $$g^n(g^{n_1})^{(t_1)}\dots(g^{n_k})^{(t_k)}\equiv b.$$

If $g^n(g^{n_1})^{(t_1)}\dots(g^{n_k})^{(t_k)}\equiv b.$ By Lemma~\ref{L2}, order of $g$ is at most 1. So we have $g(z)=e^{P(z)}$ by Lemma \ref{lm2.8}, where $P$ is a polynomial with degree at most $1$. Thus $g(\xi)=e^{c\xi+d}$, where $c$ is a nonzero constant. It implies that
$$g^{n}(\xi )(g^{n_1}(\xi ))^{(t_1)}\dots(g^{n_k}(\xi ))^{(t_k)}=(n_1c)^{t_1}\dots(n_kc)^{t_k}e^{(n+\sum\limits_{j=1}^{k}n_j)c\xi +(n+\sum\limits_{j=1}^{k}n_j)d}\equiv b.$$
This is a contradiction. Hence
\begin{align}\label{m11}
 g^n(g^{n_1})^{(t_1)}\dots(g^{n_k})^{(t_k)}\ne b.
\end{align}
We consider two subcases as following:

\noindent{\bf Case 2.1.} $g$ is a meromorphic function. Since the condition 
$$n_j\geqslant t_j, n+\sum_{j=1}^{k}n_j\geqslant \sum_{j=1}^{k}t_j+3,$$
we get that $g^n(g^{n_1})^{(t_1)}\dots(g^{n_k})^{(t_k)}- b$ has a zero by Lemma \ref{lm2.4} and Lemma \ref{lm2.5}. It contradicts with (\ref{m11}).

\noindent{\bf Case 2.2.} If $g$ a transcendental entire function (note that $g\ne 0$). The first, $n=0, k=1, n_1=t_1+1$ (see \cite{HW}) and
 $n_1\geqslant t_1+2$ (by Lemma \ref{lm2.4} and Lemma \ref{lm2.5}),  then $(g^{n_1})^{t_1}-b$ has a zero.
The second,  $ n\geqslant 1 \hspace{0.2cm} \text{or} \hspace{0.2cm} k\geqslant 2, n_j\geqslant t_j, 
n+\sum_{j=1}^{k}n_j\geqslant \sum_{j=1}^{k}t_j+2$, by Lemma \ref{lm2.4}, we get that
 $g^n(g^{n_1})^{(t_1)}\dots(g^{n_k})^{(t_k)}- b$ has a zero. It contradicts with (\ref{m11}).  If $g$ is a polynomial, since $k,n,n_j,t_j$ satisfy the assumption of Theorem \ref{th3}, then $g^n(g^{n_1})^{(t_1)}\dots(g^{n_k})^{(t_k)}- b$ has a zero. This is a contradiction by (\ref{m11}). Hence, Theorem \ref{th3} is proved.
\end{proof}

\begin{proof}[Proof of Theorem \ref{th2}]
Put $$\mathcal F=\{g_{\omega}(z)=f(z+\omega),\omega\in\C\}, \ z\in D=\Delta,$$
where $\Delta$ is the unit disk. Using Theorem \ref{th3}, we have the family $\mathcal F$ is normal in $D$. Hence, there exists a constant $M>0$ such that
$$ f^{\#}(\omega)=\dfrac{|f'(\omega)|}{1+|f(\omega)|^2}=\dfrac{|g'_{\omega}(0)|}{1+|g_{\omega}(0)|^2}\leqslant M,$$
for all $\omega\in\C$. By Lemma \ref{L2}, order of $f$ is at most 1. Since the condition 
$$f^{n+n_1+\dots+n_k}=a\rightleftharpoons f^n(f^{n_1})^{(t_1)}\dots(f^{n_k})^{(t_k)}=b,$$
$f$ must be a transcendental entire and
\begin{align}\label{ct33}
\dfrac{f^n(f^{n_1})^{(t_1)}\dots(f^{n_k})^{(t_k)}-b}{f^{n+n_1+\dots+n_k}-a}=e^{\alpha(z)}.
\end{align}
From (\ref{ct33}), we have $$T(r, e^{\alpha(z)})=O(T(r,f)).$$ 
So $\sigma(e^{\alpha})\leqslant \sigma(f)\leqslant 1$. It implies that $\alpha(z)$ must be a polynomial and $\deg(\alpha)\leqslant 1.$ Since $f$ is a transcendental entire, 
$M(r,f)\to\infty$ as $r\to\infty$. Let $$M(r_n,f)=|f(z_n)|,$$ where $z_n=r_ne^{i\theta_n}$, $\theta_n \in [0,2\pi)$, $|z_n|=r_n.$ 
We see that
\begin{align}\label{ct34}
\underset{{{r}_{n}}\to \infty }{\mathop{\lim }}\,\frac{1}{|f(z{}_{n})|}=\underset{{{r}_{n}}\to \infty }{\mathop{\lim }}\,\frac{1}{M({{r}_{n}},f)}=0.
\end{align}
By Lemma \ref{lm2.7}, there exists a set $F\subset \R_{+}$ of finite logarithmic measure such that 
\begin{align}\label{ct35}
\dfrac{f^{(m)}(z)}{f(z)}=\biggl(\dfrac{\nu(r,f)}{z}\biggl)^{m}(1+o(1))
\end{align}
holds for all $m\geqslant 1$ and $r\not\in F.$
By computing simply, we have
\begin{align}\label{ct36}
(f^n)^{(k)}=\sum c_{m_{0},m_{1},\dots, m_{k}}f^{m_{0}}(f')^{m_1}\dots (f^{(k)})^{m_k},
\end{align}
where $c_{m_{0},m_{1},\dots,m_{k}}$ are constants, and $m_0, m_1,\dots, m_k$ are nonnegative integers such that $m_0+m_1+\dots+m_k=n$, $\sum_{j=1}^{k}jm_j=k.$

From (\ref{ct36}), we have
$$\dfrac{(f^n)^{(k)}}{f^n}=\sum {c_{m_{0},m_{1},\dots, m_{k}}}\dfrac{f^{m_{0}}}{f^{m_0}}\dfrac{(f')^{m_1}}{f^{m_1}}\dots \dfrac{(f^{(k)})^{m_k}}{f^{m_k}}.$$
It implies
\begin{align}\label{ct37}
\dfrac{(f^n)^{(k)}(z_j)}{f^n(z_j)}&=\sum {c_{m_{0},m_{1},\dots, m_{k}}}\dfrac{(f')^{m_1}(z_j)}{f^{m_1}(z_j)}\dots \dfrac{(f^{(k)})^{m_k}(z_j)}{f^{m_k}(z_j)}\\
&=\sum {c_{m_{0},m_{1},\dots, m_{k}}}\biggl(\dfrac{\nu(r_j,f)}{z_j}\biggl)^{m_1+\dots+m_k}(1+o(1)).\notag
\end{align}
From (\ref{ct33}), we have 
\begin{align}\label{ct38}
\frac{\frac{(f^{n_1})^{(t_1)}\dots(f^{n_k})^{(t_k)}}{f^{n_1}\dots f^{n_k}}-\frac{b}{f^{n+n_1+\dots+n_k}}}{1-\frac{a}{f^{n+n_1+\dots+n_k}}}=e^{\alpha (z)}.
\end{align}
Apply (\ref{ct37}) to (\ref{ct38}), (\ref{ct35}) and Lemma \ref{lm2.6}, we get
\begin{align}\label{ct39}
|\alpha(z_n)|&=|\log e^{\alpha(z_n)}|=\Big|\log \frac{\frac{(f^{n_1})^{(t_1)}(z_n)\dots(f^{n_k})^{(t_k)}(z_n)}{f^{n_1}(z_n)\dots f^{n_k}(z_n)}-\frac{b}{f^{n+n_1+\dots+n_k}(z_n)}}{1-\frac{a}{f^{n+n_1+\dots+n_k}(z_n)}}\Big|\\
&\leqslant O(\log \nu(r_n,f))+O(\log r_n) +O(1)\notag\\
&=O(\log r_n),\notag
\end{align}
as $r_n\to \infty.$ From (\ref{ct39}), we obtain that $\alpha(z)$ is a constant because $\alpha(z)$ is a polynomial. By the equality (\ref{ct33}), we have
$$\dfrac{f^n(f^{n_1})^{(t_1)}\dots(f^{n_k})^{(t_k)}-b}{f^{n+n_1+\dots+n_k}-a}=c.$$
If $a=b$, we now show the existence of $\xi_0$ such that
$$f^n(\xi_0)(f^{n_1})^{(t_1)}(\xi_0)\dots(f^{n_k})^{(t_k)}(\xi_0)=b.$$
Since $f$ is a transcendental entire function, so if $n=0, k=1, n_1=t_1+1$ (see  \cite{HW}) and $n_1\geqslant t_1+2$
 (by Lemma \ref{lm2.4} and Lemma \ref{lm2.5}), then $f^n(f^{n_1})^{(t_1)}\dots(f^{n_k})^{(t_k)}-b$ has infinite zeros. Hence, $\xi_0$ definitely exists. If $n\geqslant1 \hspace{0.2cm} \text{or}\hspace{0.2cm}  k\geqslant 2$, from the condition $n+\sum_{j=1}^{k}n_j\geqslant \sum_{j=1}^{k}t_j+2$ and Lemma \ref{lm2.4}, we have 
$$f^n(f^{n_1})^{(t_1)}\dots(f^{n_k})^{(t_k)}-b$$  
have infinite zeros. Then we obtain the number $\xi_0$ satisfying 
$$f^n(\xi_0)(f^{n_1})^{(t_1)}(\xi_0)\dots(f^{n_k})^{(t_k)}(\xi_0)=b$$
with multiplicity $m\geqslant1$. By hypothesis, we see that $\xi_0$ is  a zero of $f^{n+n_1+\dots+n_k}-b$ with multiplicity $m.$ It implies that  

$$1=\dfrac{f^n(\xi_0)(f^{n_1})^{(t_1)}(\xi_0)\dots(f^{n_k})^{(t_k)}(\xi_0)-b}{f^{n+n_1+\dots+n_k}(\xi_0)-b}=c.$$
So $f^n(f^{n_1})^{(t_1)}\dots(f^{n_k})^{(t_k)}=f^{n+n_1+\dots+n_k},$ then $f$ has not zeros and order of $f$ has at most 1. It implies that $f=c_1e^{tz}$, where $c_1$ and $t$ are nonzero constants and $t$ is satisfied $(tn_1)^{t_1}\dots(tn_k)^{t_k}=1$.
\end{proof}

\noindent{\bf Acknowledgement.} The authors wish to thank the managing editor and referees for their very
helpful comments and useful suggestions. The present research was supported by the Vietnam National Foundation for Science and Technology Development
(NAFOSTED). Both authors would like to thank the Vietnam Institute for Advanced Study in Mathematics for financial support.

\end{document}